\documentclass[11pt]{article}

\usepackage[margin=1in]{geometry}
\usepackage[T1]{fontenc}
\usepackage{setspace}
\usepackage{amsmath,amsfonts,amssymb,mathtools}

\usepackage{amsthm}

\usepackage{times}
\usepackage[varg]{txfonts}
\usepackage{bm}

\usepackage{tikz}
\usetikzlibrary{patterns}

\usepackage{booktabs}
\usepackage{multirow}
\usepackage{graphicx}
\usepackage{longtable}
\usepackage{threeparttable}

\usepackage{enumitem}
\setlist{nosep,leftmargin=*}
\setlist[enumerate]{label=\arabic*.}
\usepackage{xcolor}
\usepackage[font=small,labelfont=bf]{caption}

\usepackage{algorithm}
\usepackage{algpseudocode}

\usepackage{hyperref}
\hypersetup{colorlinks=true,linkcolor=blue!60!black,citecolor=blue!60!black,urlcolor=blue!60!black}
\usepackage{natbib}
\newtheorem{proposition}{Proposition}

\newtheorem{remark}{Remark}

\newcommand{\appref}[1]{Appendix~\ref{#1}}

\title{\textbf{Exact and Matheuristic Methods for the\\
       Generalized Multiple Strip Packing Problem}}

\author{
Hyunwoo Lee\thanks{Corresponding author. Email: \texttt{hyunwoolee@vt.edu}}\\
{\small Grado Department of Industrial and Systems Engineering,
  Virginia Tech}
\and
Taesu Cheong\thanks{Email: \texttt{tcheong@korea.ac.kr}}\\
{\small School of Industrial and Management Engineering,
  Korea University}
}

\date{}

\begin{document}
\maketitle

\begin{abstract}\noindent
We study the Generalized Multiple Strip Packing Problem (GMSPP), in
which rectangular items are packed without overlap into open-ended
strips of different widths and an item's dimensions may depend on
the strip it is assigned to---a recently introduced model bridging
cutting and packing and scheduling on heterogeneous platforms.
While the single-strip problem possesses a mature exact toolbox,
the GMSPP has so far been approached only heuristically. We close
this gap with the first exact framework for the problem, covering
the total-area and makespan objectives at once, and an extensive
comparative study of exact and matheuristic methods. The framework
comprises a combinatorial Benders decomposition whose
variable-height master requires new $H_i$-aware conditional cuts,
which we prove valid; a packing-aware LP dual bound proven to
dominate the two previously available LP relaxations, and
empirically stronger than the scheduling relaxation of
\citet{vasilyev2023} by up to $58\%$ on tall-item families; and a
certified matheuristic
coupling adaptive large neighborhood search with exact components,
so that every run returns a machine-independent optimality
certificate. The Benders decomposition proves the first optima
reported for the GMSPP and remains the strongest method up to
$n=800$; beyond that frontier the matheuristic delivers proven
optima at $n=2532$ and $m=8$. All instances and code are released.
\end{abstract}

\section{Introduction}\label{sec:intro}

In the strip packing problem (SPP), rectangular items must be packed
without overlap into a single strip of fixed width and infinite
height, minimizing the height used. The SPP is a fundamental model in
cutting and packing---stock rolls of paper, steel, or glass cut with
minimum waste---and, read sideways, in scheduling: a strip is a bank
of processors, an item a task requiring a contiguous block of
processors for a duration. Many industrial settings, however, feature
\emph{several} strips that are not identical: rolls of different
widths run in parallel; computing clusters mix nodes of different
capacities; container terminals operate quays of different lengths.
\citet{vasilyev2023} recently formalized this setting as the
generalized multiple strip packing problem (GMSPP), in which each of
$m$ strips has its own width and an item's dimensions may depend on
the strip it is assigned to, and proposed integer formulations, two
heuristic algorithms, and a scheduling-based lower bound.

What is missing is exactness. No exact algorithm is available for the
GMSPP, so solution quality can only be read against a heuristic upper
bound and a scheduling lower bound; on mid-size instances the gap
between them exceeds $40\%$, and in places $80\%$
\citep{vasilyev2023}, with no way to tell how much of it is primal
and how much dual. By contrast, the
single-strip SPP possesses a mature exact toolbox, most notably the
combinatorial Benders decomposition BLUE of \citet{cote2014}, which packs items
into unit-width slices in a master problem and restores vertical
feasibility in a slave (``y-check''). Whether this machinery survives
the generalization to multiple heterogeneous strips---with an
assignment layer interacting with packing, per-strip height variables
replacing a fixed trial height, and strip-dependent item
dimensions---is the question this paper answers.

We answer it constructively, for two canonical objectives at once:
minimizing the total used area (the cutting-and-packing reading) and
minimizing the makespan (the scheduling reading, and the objective of
\citealt{vasilyev2023}). The paper is best read as bridging two
literatures that have not met: the sophisticated exact machinery
developed for the single strip, and the heuristic-only state of the
art for multiple heterogeneous strips. Our contributions are as
follows.

\begin{enumerate}
\item \textbf{The first exact framework for the GMSPP.} We extend
  the normal-position formulation and the y-check decomposition of
  \citet{cote2014} to multiple heterogeneous strips with
  strip-dependent item dimensions. The slave problem and all cuts are
  pure feasibility devices, so a single engine handles both
  objectives---and any objective monotone in the strip heights.
  In BLUE the strip height is a fixed trial value, and the cuts are
  valid only at that height; in the GMSPP the per-strip heights are
  decision variables, coupled through the assignment and, under the
  makespan, through the objective. The fixed-height recognition
  scheme therefore does not transfer: our master keeps the $H_i$ as
  variables, which requires a new, $H_i$-aware conditional form of
  the combinatorial Benders cuts; we prove its validity and
  transplant the interval lifting of \citet{cote2014} per strip. The resulting
  algorithm (BendM) proves, to our knowledge, the first optima
  reported for the GMSPP.
\item \textbf{Dual bounds with a complete strength map.} We
  strengthen the LP relaxation of the master with height-linking
  inequalities and prove that the resulting bound (LP-PC) dominates
  the LP relaxation of the big-$M$ formulation, strictly on some
  instances, as well as the LP relaxation of the scheduling bound of
  \citet{vasilyev2023}, whose integer version is provably
  incomparable with LP-PC. Computationally the strength map is
  regime-structured: the bounds coincide on area-dominated families,
  LP-PC wins by up to $58\%$ on the tall-item families, the regime in
  which published gaps for the GMSPP are largest, and on
  wide-strip families a trivial height bound beats both LP-based
  bounds---so we operate a three-member portfolio whose every member
  earns its place.
\item \textbf{A certified matheuristic.}
  We bring neighborhood search to this literature for the first
  time, as a \emph{matheuristic} in the sense of
  \citet{fischetti2010matheuristics}---a heuristic that delegates
  subproblems to exact optimization: adaptive large neighborhood
  search over assignments, local search, and an exact
  constraint-programming endgame that re-packs each strip, coupled
  with the bound portfolio so that every run returns a
  machine-independent optimality certificate. Beyond the reach of
  the exact methods it delivers proven optima at $n=2532$ and
  $m=8$, where the only previously available bound is off by
  $90\%$.
\item \textbf{An honest equal-budget comparison that settles the
  method question.} Under a common $1200$-second budget we compare
  the plain big-$M$ model (stated by \citealt{vasilyev2023}, who
  judge it tractable only at very small sizes), a bound-injected
  variant,
  BendM, and the matheuristic package, over both objectives and $n$
  from $20$ to $2532$. The verdict was not the one we anticipated:
  the combinatorial Benders machinery dominates---primal and
  dual---on every family up to $n=800$, while the plain big-$M$
  collapses already at $n=20$ and the package's certified gaps stay
  within $1.3$--$2.5$ percentage points of BendM's on CLASS. We report the
  resulting exact/matheuristic frontier explicitly, by size and by
  objective. Instances and code are released in full.
\end{enumerate}

The paper is organized as follows. Section~\ref{sec:related} reviews
related work. Section~\ref{sec:problem} defines the problem and makes
explicit the two modeling lenses---cutting-and-packing with fixed
item dimensions, scheduling with strip-dependent ones---that organize
our study. Section~\ref{sec:exact} develops the exact methods and
Section~\ref{sec:matheuristic} the matheuristic.
Section~\ref{sec:experiments} reports the computational study, and
Section~\ref{sec:conclusion} concludes.

\section{Related work}\label{sec:related}

\paragraph{Exact methods for single-strip packing.}
Exact algorithms for the SPP include the branch-and-bound of
\citet{martello2003}, combinatorial branch-and-bounds building
solutions item by item \citep{alvarez2009,kenmochi2009,boschetti2010,
arahori2012}, and MILP- and SAT-based approaches
\citep{castro2011,soh2010}. The state of the art
is the combinatorial Benders decomposition BLUE of \citet{cote2014}
(logic-based rather than classical LP-dual Benders): a master packs
items into unit-width slices at normal positions
\citep{herz1972,christofides1977} under column-load
constraints---the relaxation of $P|cont|C_{\max}$ (itself equivalent
to the SPP, cf.\ \citealt{bladek2015}) in which the slices of an
item need not be vertically aligned---and a slave (y-check) decides whether
vertical positions exist realizing a trial height; infeasibility is
translated into combinatorial Benders cuts \citep{codato2006},
strengthened by minimal infeasible subsets and interval lifting.
\citet{delorme2016} survey the related one-dimensional relaxations,
and the meet-in-the-middle principle of \citet{cote2018} later
refined the normal-position patterns used in such formulations.

Our exact framework is a descendant of this line;
Section~\ref{sec:exact} details what changes---and what must be
re-proven---when the height becomes a vector of decision variables
across heterogeneous strips. Our comparison of LP relaxations in
\S\ref{sec:normal} parallels a classical theme in machine scheduling,
where time-indexed formulations are known to provide stronger LP
bounds than disjunctive big-$M$ ones
\citep{dyer1990,sousa1992,queyranne1994}; to our knowledge no such
dominance result was available for two-dimensional packing
formulations, and ours hinges on a strengthening inequality that has
no single-machine counterpart.

\paragraph{Multiple strips.}
The multiple strip packing problem with identical strips was studied
mainly through the lens of approximation algorithms
\citep{zhuk2006,bougeret2010,ye2011}, motivated by scheduling parallel
jobs on platforms; polynomial algorithms with guarantees of $2$ and
$5/2$ are known, and \citet{bladek2015} clarified the relation between
contiguous and non-contiguous variants. Heuristics for heterogeneous
strips appear in \citet{vasilyev2021}. The generalized problem with
strip-dependent item dimensions and its applications---heterogeneous
cloud scheduling, berth allocation with multiple quays---are due to
\citet{vasilyev2023}, who propose a big-$M$ linearization (stated but
not computationally evaluated in their study, as they judge it
tractable only for very small instances), a level-based
restriction providing upper bounds, two heuristics (a skyline best-fit
with randomized local search, and a two-stage method assigning items
by a ``liquified'' minimum-makespan scheduling MIP), and a lower bound
from the same scheduling relaxation. Before the present paper, no
exact method and no packing-aware dual bound was available for the
GMSPP.

\paragraph{Matheuristics.}
Hybrid schemes in which a metaheuristic delegates subproblems to exact
solvers are by now standard \citep{fischetti2010matheuristics}; large
neighborhood search itself originated in constraint programming
\citep{shaw1998}, and adaptive operator selection follows
\citet{ropke2006}. ALNS has recently been applied to packing
variants such as circle bin packing \citep{he2021} and
two-dimensional packing with conflicts \citep{zeng2021}. Closest to
our setting, \citet{horn2023} apply ALNS to a shelf-restricted
multi-strip problem with strip-dependent item widths and time
windows, arising in tool coating; their packings are however
level-structured (items on shelves), and no dual bounds accompany
the heuristic. To our knowledge no neighborhood-search method has
addressed the GMSPP with free (non-level) packings, and none
returns optimality certificates. Our Phase-3 endgame---an exact CP
re-optimization of each strip under a frozen assignment---is an
instance of fix-and-optimize tailored to the two-level structure of
GMSPP (Section~\ref{sec:diagnostic}), with its division of labor
validated ex post (Section~\ref{sec:ablation}): the exactness is
spent at the packing level, where it certifies per-strip optimality
and guarantees that the final packing quality does not depend on
the heuristic packer.

\paragraph{Positioning.}
In the typology of \citet{wascher2007} the GMSPP is an open dimension
problem with multiple containers; under fixed item dimensions it can
be read as the open-dimension counterpart of variable-sized bin
packing. To our knowledge, the combination of (i) exact Benders
machinery on multiple heterogeneous open-dimension containers,
(ii) strip-dependent item dimensions, and (iii) a certified-gap
matheuristic has not previously been studied.

\section{Problem definition}\label{sec:problem}

We are given a set $\mathcal{J}=\{1,\dots,n\}$ of rectangular items and
a set $\mathcal{I}=\{1,\dots,m\}$ of strips. Strip $i$ has width $W_i$
and infinite height. Item $j$, when assigned to strip $i$, occupies a
rectangle of integer width $w_{ij}$ and integer height $h_{ij}$
\citep{vasilyev2023}. Each item must be packed exactly once into one of
the strips, without overlap and without rotation, with edges parallel
to the strip borders. An item $j$ may be packed into strip $i$ only if
$w_{ij}\le W_i$; we assume every item admits at least one such strip.
Let $H_i$ denote the height of the packing in strip $i$, i.e., the
smallest height containing all items assigned to it. We consider two
objectives:
\begin{align}
  \text{(GMSPP-A)}\quad & \min\ \textstyle\sum_{i\in\mathcal{I}} W_i H_i,
  \label{eq:obj-area}\\
  \text{(GMSPP-M)}\quad & \min\ \max_{i\in\mathcal{I}} H_i.
  \label{eq:obj-makespan}
\end{align}
Objective~\eqref{eq:obj-area} minimizes the total used area, i.e., the
trim loss summed over the strips; objective~\eqref{eq:obj-makespan}
minimizes the makespan. Both problems are strongly NP-hard, since
for $m=1$ they reduce to the strip packing problem (SPP), which is
strongly NP-hard by reduction from one-dimensional bin packing
\citep{garey1979,martello2003}.

\subsection{Two modeling lenses}\label{sec:lenses}

The general model above sits at the junction of two research
traditions, and we will keep the distinction explicit throughout.

\paragraph{Cutting-and-packing lens (fixed dimensions).}
In the cutting-and-packing literature, item dimensions are intrinsic:
$w_{ij}=w_j$ and $h_{ij}=h_j$ for all $i$. Heterogeneity then resides
solely in the containers, as in variable-sized bin packing. Under this
lens, GMSPP is a \emph{multiple strip packing problem with
heterogeneous widths}: stock rolls of different widths processed in
parallel, with objective~\eqref{eq:obj-area} measuring total material
consumption. In the typology of \citet{wascher2007}, the problem is an
open dimension problem (ODP) with several containers, each having one
variable dimension.

\paragraph{Scheduling lens (strip-dependent dimensions).}
In parallel task scheduling on heterogeneous platforms, strip $i$ is a
computing node offering $W_i$ contiguous processors, and item $j$ is a
task requiring $w_{ij}$ processors for $h_{ij}$ time units when run on
node $i$ \citep{vasilyev2023}. Dimensions naturally depend on the node:
a faster node shortens $h_{ij}$, a different architecture changes
$w_{ij}$. Berth allocation with multiple quays gives a second
interpretation, where handling time depends on the quay's equipment
while vessel length is fixed. Practically motivated dependence is
structured: \emph{speed scaling} ($w_{ij}=w_j$, $h_{ij}=h_j\tau_i$ for
node-speed factors $\tau_i$), and \emph{area-conserving reshaping}
($w_{ij}h_{ij}\approx a_j$, a task with fixed total work), which is the
Type~I setting of \citet{vasilyev2023}. Their Type~II setting draws
$(w_{ij},h_{ij})$ arbitrarily; we include it for comparability.

All methods in this paper are stated for the general strip-dependent
model; the fixed-dimension case is obtained by dropping the index $i$
from the item dimensions, and none of our algorithms require
consistency assumptions across strips.

\begin{remark}[Cost-weighted objectives]\label{rem:costs}
Our exact framework supports any objective of the form
$\sum_i C_i W_i H_i$ with per-unit-area costs $C_i>0$, of which
\eqref{eq:obj-area} is the case $C_i\equiv 1$; heterogeneous-cost
variants require no change to any component below. More
generally, Section~\ref{sec:exact} shows that the decomposition is
valid for any objective that is monotone non-decreasing in the strip
heights $(H_1,\dots,H_m)$, which covers \eqref{eq:obj-makespan} as
well. In this paper we focus on the two canonical
objectives~\eqref{eq:obj-area}--\eqref{eq:obj-makespan}.
\end{remark}

\begin{remark}[Relation between the objectives]\label{rem:relation}
Neither objective subsumes the other: \eqref{eq:obj-area} is separable
across strips, whereas \eqref{eq:obj-makespan} requires an auxiliary
maximum variable and couples the strips. The two objectives also pull
solutions in opposite directions---\eqref{eq:obj-area} concentrates
items on few strips when profitable, while \eqref{eq:obj-makespan}
balances load---so computational behavior differs markedly
(Section~\ref{sec:experiments}).
\end{remark}

\section{Exact methods}\label{sec:exact}

This section presents our exact framework. We first recall the
linearized big-$M$ formulation of \citet{vasilyev2023} and explain why
its LP relaxation is weak (\S\ref{sec:bigm}). We then introduce the
normal-position master formulation and the dual bound LP-PC
(\S\ref{sec:normal}), the lower-bound-enhanced big-$M$
(\S\ref{sec:bigmle}), and the Benders decomposition BendM
(\S\ref{sec:bendm}). Throughout, the objective is either
\eqref{eq:obj-area} or \eqref{eq:obj-makespan}; we write the models
with per-strip height variables $H_i$ and note that only the objective
row differs between the two cases.

\subsection{Linearized big-$M$ formulation}\label{sec:bigm}

% (adapted from Vasilyev et al. 2023, eqs. (1)-(12); our tightened
%  linking constraints and per-strip big-$M$ constants)
Binary $z_{ij}$ assigns item $j$ to strip $i$; continuous
$x_{ij},y_{ij}\ge 0$ give the bottom-left corner of $j$ on strip $i$;
binaries $\ell_{jk}, b_{jk}$ encode the left-of / below relations.
With $M_i=\sum_{j:\,w_{ij}\le W_i} h_{ij}$,
\begin{align}
 \min\quad & \textstyle\sum_i W_i H_i \quad\text{or}\quad
             \min\ H^{\max},\ H^{\max}\ge H_i\ \forall i \notag\\
 \text{s.t.}\quad
 & \textstyle\sum_{i:\,w_{ij}\le W_i} z_{ij}=1, && j\in\mathcal{J},
   \label{eq:bm-assign}\\
 & x_{ij}+w_{ij}\le W_i, && \forall (i,j),\label{eq:bm-fit}\\
 & y_{ij}+h_{ij}\le H_i + h_{ij}(1-z_{ij}), && \forall (i,j),
   \label{eq:bm-height}\\
 & x_{ij}\le (W_i-w_{ij})\,z_{ij},\quad
   y_{ij}\le (M_i-h_{ij})\,z_{ij}, && \forall (i,j),
   \label{eq:bm-link}\\
 & x_{ij}+w_{ij}\le x_{ik}+W_i\,(3-\ell_{jk}-z_{ij}-z_{ik}),
   && \forall (i,j,k),\label{eq:bm-nox}\\
 & y_{ij}+h_{ij}\le y_{ik}+M_i\,(3-b_{jk}-z_{ij}-z_{ik}),
   && \forall (i,j,k),\label{eq:bm-noy}\\
 & \ell_{jk}+\ell_{kj}+b_{jk}+b_{kj}=1, && j<k.\label{eq:bm-rel}
\end{align}
Constraint \eqref{eq:bm-height} uses the tight coefficient $h_{ij}$:
when $z_{ij}=0$, \eqref{eq:bm-link} forces $y_{ij}=0$ and the
constraint is vacuous, while in the LP relaxation it still yields
$H_i\ge h_{ij}z_{ij}$.

The LP relaxation of \eqref{eq:bm-assign}--\eqref{eq:bm-rel} is weak:
setting all $\ell,b$ variables to $1/4$ (say) satisfies
\eqref{eq:bm-rel} and deactivates the disjunctive constraints
\eqref{eq:bm-nox}--\eqref{eq:bm-noy} through their big-$M$ terms, so the
relaxation essentially reduces to the height cuts implied by
\eqref{eq:bm-height}. We refer to this formulation handed directly
to the MIP solver as \emph{BigM}. We quantify its weakness, and
repair it, in \S\ref{sec:bigmle}.

\subsection{Normal-position master and the LP-PC bound}
\label{sec:normal}

Following the single-strip decomposition of \citet{cote2014}, we
restrict the $x$-coordinates to \emph{normal positions}---the
dominance-based discretization of \citet{herz1972} and
\citet{christofides1977}. For strip $i$
and item $j$, let $\mathcal{W}_i(j)$ be the set of positions
$p\in\{0,\dots,W_i-w_{ij}\}$ expressible as sums of widths
$w_{ik}\,(k\ne j)$ of the other items, evaluated with the dimensions
they take on strip $i$; these sets are computed by subset-sum dynamic
programming.
Binary $x_{ijp}$ places item $j$ at position $p$ on strip $i$:
\begin{align}
 \min\quad & \textstyle\sum_i W_i H_i \quad\text{or}\quad H^{\max}
 \notag\\
 \text{s.t.}\quad
 & \textstyle\sum_{i}\sum_{p\in\mathcal{W}_i(j)} x_{ijp}=1,
   && j\in\mathcal{J},\label{eq:np-assign}\\
 & \textstyle\sum_{j}\sum_{p\in\mathcal{W}_i(j,q)} h_{ij}\,x_{ijp}
   \le H_i, && i\in\mathcal{I},\ q\in\{0,\dots,W_i-1\},
   \label{eq:np-colload}
\end{align}
where $\mathcal{W}_i(j,q)$ is the subset of positions at which item
$j$ covers column $q$. Constraints \eqref{eq:np-colload} state that
the total height of the items covering any unit-width column of strip
$i$ cannot exceed $H_i$; this is the multi-strip generalization of the
$P|cont|C_{\max}$ relaxation of \citet{cote2014}.

In the LP relaxation, a tall narrow item may spread fractionally over
many positions, diluting its column load far below its height. We
therefore add the \emph{height-linking} inequalities
\begin{equation}\label{eq:link}
  h_{ij}\textstyle\sum_{p\in\mathcal{W}_i(j)} x_{ijp} \;\le\; H_i,
  \qquad i\in\mathcal{I},\ j\in\mathcal{J}.
\end{equation}

\begin{proposition}\label{prop:link}
Inequalities \eqref{eq:link} are satisfied by every integer-feasible
solution of \eqref{eq:np-assign}--\eqref{eq:np-colload}; the optimal
integer value is unchanged, and only the LP relaxation is tightened.
\end{proposition}
\begin{proof}
If item $j$ is packed at some position $p$ on strip $i$, any column it
covers already carries load at least $h_{ij}$ in
\eqref{eq:np-colload}, so $H_i\ge h_{ij}$; otherwise the left-hand
side of \eqref{eq:link} is $0$.
\end{proof}

We denote by \emph{LP-PC} the LP relaxation of
\eqref{eq:np-assign}--\eqref{eq:np-colload} with \eqref{eq:link}. It
is solvable in seconds even for $n\approx 200$ and anchors the
dual-bound portfolio used throughout this paper
(\S\ref{sec:relax-exp}). LP-PC compares favorably with the
LP relaxation of the big-$M$ model: the only mechanism by which the
latter bounds $H_i$ from below is $H_i\ge h_{ij}z_{ij}$ from
\eqref{eq:bm-height}, because the disjunctive constraints
\eqref{eq:bm-nox}--\eqref{eq:bm-noy} are deactivated by fractional
relation variables (e.g., $\ell_{jk}=b_{jk}=1/4$ satisfies
\eqref{eq:bm-rel} while leaving slack at least $0.75\,M_i$).
Inequality \eqref{eq:link} reproduces exactly this mechanism, while
the column loads \eqref{eq:np-colload} additionally capture area
accumulation, which the big-$M$ LP misses entirely.

\begin{remark}[The linking inequalities are what dominance hinges on]
\label{rem:incomparable}
Without \eqref{eq:link}, the two LP relaxations are incomparable.
On a strip of width $10$ with one $1\times 10$ item and nine
$1\times 1$ items, the tall item spreads its mass over the ten normal
positions and the column-load LP evaluates to the area bound
$19/10=1.9$, whereas the big-$M$ LP yields $H\ge h_j z_j=10$.
Conversely, for ten $1\times 1$ items on a strip of width $2$, the
column-load LP equals $5$ while the big-$M$ LP equals $1$: fractional
relation variables ($\ell_{jk}=b_{jk}=1/4$) deactivate
\eqref{eq:bm-nox}--\eqref{eq:bm-noy}, and no mechanism accumulates
area. With \eqref{eq:link} added, the first gap closes and dominance
becomes one-sided, as we now prove.
\end{remark}

\begin{proposition}\label{prop:dominance}
For both objectives, the optimal value of LP-PC is at least the
optimal value of the LP relaxation of
\eqref{eq:bm-assign}--\eqref{eq:bm-rel}, and the dominance is strict
on some instances.
\end{proposition}

\begin{proof}
Let $(\bar x,\bar H)$ be an optimal LP-PC solution and write
$\sigma_{ij}=\sum_{p}\bar x_{ijp}$. We construct a point of the big-$M$
LP with the same height vector $\bar H$, hence the same objective
value; this proves the inequality. Set $z_{ij}=\sigma_{ij}$,
$x_{ij}=0$, $\ell_{jk}=\ell_{kj}=0$ and $b_{jk}=b_{kj}=1/2$ for all
pairs, so \eqref{eq:bm-assign}, \eqref{eq:bm-fit},
\eqref{eq:bm-rel}, and the $x$-part of
\eqref{eq:bm-link}--\eqref{eq:bm-nox} hold trivially (the latter
because $3-z_{ij}-z_{ik}\ge 1$). It remains to choose the $y$
variables so that \eqref{eq:bm-height}, the $y$-part of
\eqref{eq:bm-link}, and \eqref{eq:bm-noy} hold.

Call item $j$ \emph{dominant} on strip $i$ if $h_{ij}>M_i/2$; since
$M_i$ is the total height of all items fitting strip $i$, at most one
dominant item $j^*_i$ exists per strip. With the chosen relation
variables, \eqref{eq:bm-noy} for the ordered pair $(j,k)$ on strip
$i$ reads $y_{ij}+h_{ij}\le y_{ik}+M_i(5/2-z_{ij}-z_{ik})$, whose
big-$M$ term is at least $M_i/2$. Define
\[
  y_{ik}=\max\Bigl(0,\;
    h_{ij^*_i}-M_i\bigl(5/2-z_{ij^*_i}-z_{ik}\bigr)\Bigr)
  \quad (k\ne j^*_i),\qquad y_{ij^*_i}=0,
\]
and $y\equiv 0$ on strips without a dominant item.

\emph{Constraint \eqref{eq:bm-noy}.} For $(j^*_i,k)$ the definition
of $y_{ik}$ gives exactly
$h_{ij^*_i}\le y_{ik}+M_i(5/2-z_{ij^*_i}-z_{ik})$. For $(j,k)$ with
$j\ne j^*_i$ and $y_{ij}=0$, the left side is $h_{ij}\le M_i/2$,
which is at most the big-$M$ term. For $(j,k)$ with $y_{ij}>0$,
substituting the definition of $y_{ij}$ reduces the claim to
$h_{ij^*_i}+h_{ij}\le M_i\,(5-z_{ij^*_i}-2z_{ij}-z_{ik})$, and since
the coefficient sum is at most $4$, the right side is at least
$M_i\ge h_{ij^*_i}+h_{ij}$.

\emph{Constraint \eqref{eq:bm-height}}, i.e.,
$\bar H_i\ge y_{ik}+h_{ik}z_{ik}$. For $k=j^*_i$ this is the linking
inequality \eqref{eq:link}. For $y_{ik}>0$ we have
$z_{ij^*_i},z_{ik}>1/2$ (otherwise the big-$M$ term exceeds
$h_{ij^*_i}$), and by \eqref{eq:link} it suffices to show
$h_{ij^*_i}z_{ij^*_i}\ge y_{ik}+h_{ik}z_{ik}$, equivalently
\[
  M_i\bigl(5/2-z_{ij^*_i}-z_{ik}\bigr)\;\ge\;
  h_{ij^*_i}\bigl(1-z_{ij^*_i}\bigr)+h_{ik}z_{ik}.
\]
Using $M_i\ge h_{ij^*_i}+h_{ik}$ and collecting terms, the difference
of the two sides is at least
$h_{ij^*_i}(3/2-z_{ik})+h_{ik}(5/2-z_{ij^*_i}-2z_{ik})
 \ge \tfrac12\,(h_{ij^*_i}-h_{ik})>0$,
since the dominant item satisfies $h_{ij^*_i}>h_{ik}$.

\emph{Constraint \eqref{eq:bm-link}, $y$-part.} When $y_{ik}>0$,
$y_{ik}\le h_{ij^*_i}-M_i/2\le M_i/2-h_{ik}\le (M_i-h_{ik})z_{ik}$
using $z_{ik}>1/2$.

Hence the constructed point is feasible with value equal to the
LP-PC optimum. Strictness follows from the second instance of
Remark~\ref{rem:incomparable}, on which \eqref{eq:link} is inactive
and LP-PC still equals $5$ against $1$.
\end{proof}

Proposition~\ref{prop:dominance} explains, and is confirmed by, the
computational comparison of Section~\ref{sec:experiments}. It also
clarifies the role of BigM-LE (\S\ref{sec:bigmle}): injecting the
LP-PC bound equips the big-$M$ model with the relaxation strength it
provably lacks.

\subsection{Relation to the liquified scheduling bound}
\label{sec:mmsp}

The only lower bound previously available for the GMSPP is the
scheduling relaxation of \citet[Prop.~1]{vasilyev2023}: item $j$ on
strip $i$ is ``liquified'' into $s_{ij}=w_{ij}h_{ij}/W_i$ units of
height, and the minimum makespan scheduling problem
\begin{equation}\label{eq:mmsp}
  Z^{\mathrm{ms}}=\min\Bigl\{\,H:\ \textstyle\sum_{j} s_{ij}v_{ij}\le H
  \ \ \forall i,\ \ \sum_{i} v_{ij}=1\ \ \forall j,\ \
  v_{ij}\in\{0,1\},\ v_{ij}=0 \text{ if } w_{ij}>W_i\Bigr\}
\end{equation}
satisfies $Z^{\mathrm{ms}}\le Z^*$ for the makespan objective. It is
solved as a small MIP in seconds. We note that the lower bounds
\emph{reported} by \citet[\S5.2]{vasilyev2023} are a strengthened
variant that additionally fixes $v_{ij}=0$ whenever $h_{ij}$ exceeds
the upper bound returned by their two-stage heuristic; throughout
this paper $Z^{\mathrm{ms}}$ denotes the pure relaxation
\eqref{eq:mmsp}, which depends on instance data only.

How does it relate to LP-PC? The answer splits by objective, and
the split itself is instructive. Under the \emph{total-area}
objective, liquification loses all its content: since
$W_i\,s_{ij}=w_{ij}h_{ij}$, the liquified objective collapses to
the sum of assigned item areas, so the bound degenerates to
$\sum_j\min_i w_{ij}h_{ij}$---the elementary area bound---which
LP-PC dominates by the aggregation argument of
Proposition~\ref{prop:mmsp-lp} below. It is the min--max structure
of the \emph{makespan} that leaves liquification a non-trivial
load-balancing problem; there the comparison is genuinely
two-sided and decomposes cleanly along two orthogonal axes.

\begin{proposition}\label{prop:mmsp-lp}
LP-PC dominates the LP relaxation of \eqref{eq:mmsp}.
\end{proposition}
\begin{proof}
Let $(\bar x,\bar H)$ be LP-PC-feasible and set
$v_{ij}=\sigma_{ij}=\sum_p \bar x_{ijp}$. Each position of item $j$ on
strip $i$ covers exactly $w_{ij}$ columns, so summing the column-load
constraints \eqref{eq:np-colload} over the $W_i$ columns of strip $i$
gives $\sum_j h_{ij}w_{ij}\sigma_{ij}\le W_i\bar H_i$; dividing by
$W_i$ yields $\sum_j s_{ij}v_{ij}\le \bar H_i$, i.e., $(v,\bar H)$ is
feasible for the LP relaxation of \eqref{eq:mmsp} with the same
makespan.
\end{proof}

The proof reveals the structural relation: \emph{the column-load
constraints are a column-by-column disaggregation of the liquified
machine loads}. The two published bounds thus refine a common base in
orthogonal directions: $Z^{\mathrm{ms}}$ adds \emph{integrality} of
the assignment to the liquified base, while LP-PC adds
\emph{geometry} (positions, per-column loads, height linking).

\begin{remark}[Integrality and geometry are incomparable]
\label{rem:mmsp-incomparable}
Neither refinement dominates the other. Geometry wins on tall narrow
items: for one $1\times 10$ and nine $1\times 1$ items on a single
strip of width $10$, $Z^{\mathrm{ms}}=1.9$ while LP-PC $\ge 10$ by
\eqref{eq:link}. Integrality wins on few large items: for a single
$10\times 10$ item and two strips of width $10$,
$Z^{\mathrm{ms}}=10$ while LP-PC $=5$, the fractional assignment
splitting the item across the strips. Both bounds are computable in seconds; together with the trivial
height bound $h^*$ of \S\ref{sec:relax-exp} they form the
dual-bound portfolio used whenever a bound is needed
(\S\ref{sec:parallel}).
\end{remark}

The resulting lattice of relaxations is summarized by
\begin{equation}\label{eq:lattice}
\begin{array}{ccccc}
 & & Z^* & & \\[2pt]
 & & \text{\small master IP
   \eqref{eq:np-assign}--\eqref{eq:np-colload}} & & \\[2pt]
 \text{LP-PC \small(geometry)} & & & &
 Z^{\mathrm{ms}} \ \text{\small(integrality)} \\[2pt]
 & & \text{\small liquified LP base} & &
\end{array}
\end{equation}
read upwards by dominance: the master integer program combines both
refinements (aggregating its column loads with integral $\sigma$
recovers \eqref{eq:mmsp}), and the Benders cuts of \S\ref{sec:bendm}
close the remaining gap between the master and $Z^*$. The big-$M$ LP
sits below LP-PC by Proposition~\ref{prop:dominance}. Under the
total-area objective the right branch of the lattice is absent---the
liquified base degenerates to the area bound, as noted at the start
of this subsection, and the dual side is carried by LP-PC
alone---so we compute $Z^{\mathrm{ms}}$ for the makespan objective
only.

\subsection{Lower-bound-enhanced big-$M$ (BigM-LE)}\label{sec:bigmle}

A pragmatic repair of \S\ref{sec:bigm} is to compute LP-PC first and
add the single constraint $\text{obj}\ge \text{LP-PC}$ to the big-$M$
model before invoking the MIP solver. We call the result BigM-LE. It
requires no structural change to the model, and
Section~\ref{sec:experiments} shows it recovers most of the
terminal-gap reduction that the stronger machinery of
\S\ref{sec:bendm} achieves on small instances---though none of its
ability to close them---at a fraction of the implementation
effort. We regard
it as a practical middle ground rather than a core contribution.

\subsection{BendM: combinatorial Benders decomposition with
$H_i$-aware cuts}
\label{sec:bendm}

BendM solves the master
\eqref{eq:np-assign}--\eqref{eq:np-colload} by branch-and-cut. At each
integer solution, the assignment and positions on each strip are fixed,
and the \emph{y-check} problem of \citet{cote2014} asks, independently
per strip, whether $y$-coordinates exist realizing the tentative height.
The y-check is objective-agnostic and strip-separable: it receives only
a set of items with fixed $x$-positions and dimensions
$(w_{ij},h_{ij})$ and a target height. If some strip fails, a cut is
added; otherwise the incumbent is feasible for GMSPP.

Figure~\ref{fig:framework} illustrates the decomposition on a
two-strip example. In panel~(a), the master has chosen assignments,
normal positions, and heights $(H_1,H_2)=(3,7)$; it checks column
loads only, so items appear as vertically sliding unit-width slices
(items 4 and 5 on strip 2). In panel~(b), the y-check must realign
each item to a single $y$-interval, and strip 2 admits no such
alignment at $H_2=7$: no arrangement of items 1--9 leaves a
contiguous free space of height $3$ in column 2 (in the arrangement
shown, the free space splits into two height-$2$ buckets, and once
$p_1$ takes one unit, nothing of height $3$ remains for
$p_2$)---although every column load equals $7$ exactly. Panel~(c)
shows why the cut must be conditional rather than a no-good: the
same configuration is feasible at $H_2=8$, so the cut introduced
below prices it at $H_2\ge 8$ instead of forbidding it.

% Framework figure: master (column loads) / y-check failure / cut.
% Instance: Cote et al. (2014) Fig.-2 NP-completeness gadget with
% B=2 plus two partition items {1,3} at x=2 (verified by CP-SAT:
% all column loads = 7, min realizable height = 8).
\begin{figure}[h!]
\centering
\newcommand{\pkgitem}[6]{% x y w h fill label
  \draw[fill=#5,draw=black,line width=.4pt]
    (#1,#2) rectangle ++(#3,#4);
  \node[font=\tiny] at (#1+#3/2,#2+#4/2) {#6};}
\begin{tikzpicture}[x=.42cm,y=.30cm]
% ---------------- (a) master view ----------------
\begin{scope}
\node[font=\small\bfseries,anchor=south west] at (-1.0,8.0)
  {(a) master: column loads $\le H_i$};
% strip 1
\draw[line width=.8pt] (0,0) -- (0,4.6) (3,0) -- (3,4.6);
\draw[line width=.8pt] (0,0) -- (3,0);
\pkgitem{0}{0}{2}{2}{gray!25}{$a$}
\pkgitem{2}{0}{1}{3}{gray!45}{$b$}
\draw[dashed] (-.3,3) node[left,font=\scriptsize] {$H_1$} -- (3.3,3);
\node[font=\scriptsize] at (1.5,-1.1) {strip 1, $W_1=3$};
% strip 2 (sliced/column view)
\begin{scope}[shift={(5.5,0)}]
\draw[line width=.8pt] (0,0) -- (0,7.8) (5,0) -- (5,7.8);
\draw[line width=.8pt] (0,0) -- (5,0);
% c0
\pkgitem{0}{0}{1}{6}{gray!45}{1}\pkgitem{0}{6}{1}{1}{gray!15}{3}
% c1
\pkgitem{1}{0}{1}{1}{gray!15}{3}\pkgitem{1}{1}{1}{1}{gray!15}{4}
\pkgitem{1}{2}{1}{3}{gray!30}{6}\pkgitem{1}{5}{1}{2}{gray!60}{9}
% c2
\pkgitem{2}{0}{1}{1}{gray!15}{3}\pkgitem{2}{1}{1}{1}{gray!15}{4}
\pkgitem{2}{2}{1}{1}{gray!15}{5}
\pkgitem{2}{3}{1}{1}{orange!50}{$p_1$}
\pkgitem{2}{4}{1}{3}{orange!70}{$p_2$}
% c3
\pkgitem{3}{0}{1}{1}{gray!15}{4}\pkgitem{3}{1}{1}{1}{gray!15}{5}
\pkgitem{3}{2}{1}{3}{gray!30}{7}\pkgitem{3}{5}{1}{2}{gray!60}{8}
% c4
\pkgitem{4}{0}{1}{6}{gray!45}{2}\pkgitem{4}{6}{1}{1}{gray!15}{5}
\draw[dashed] (-.3,7) node[left,font=\scriptsize] {$H_2$} -- (5.3,7);
\node[font=\scriptsize] at (2.5,-1.1) {strip 2, $W_2=5$};
\end{scope}
\end{scope}
% ---------------- (b) y-check at H2 = 7 ----------------
\begin{scope}[shift={(14.2,0)}]
\node[font=\small\bfseries,anchor=south west] at (-1.0,8.0)
  {(b) y-check: infeasible at $H_2=7$};
\draw[line width=.8pt] (0,0) -- (0,7.8) (5,0) -- (5,7.8);
\draw[line width=.8pt] (0,0) -- (5,0);
\pkgitem{0}{0}{1}{6}{gray!45}{1}
\pkgitem{4}{1}{1}{6}{gray!45}{2}
\pkgitem{0}{6}{3}{1}{gray!15}{3}
\pkgitem{1}{3}{3}{1}{gray!15}{4}
\pkgitem{2}{0}{3}{1}{gray!15}{5}
\pkgitem{1}{0}{1}{3}{gray!30}{6}
\pkgitem{3}{4}{1}{3}{gray!30}{7}
\pkgitem{3}{1}{1}{2}{gray!60}{8}
\pkgitem{1}{4}{1}{2}{gray!60}{9}
\pkgitem{2}{1}{1}{1}{orange!50}{$p_1$}
% two height-2 buckets at column 2
\draw[pattern=north east lines, pattern color=gray!70]
  (2,2) rectangle ++(1,1);
\draw[pattern=north east lines, pattern color=gray!70]
  (2,4) rectangle ++(1,2);
\draw[dashed] (-.3,7) node[left,font=\scriptsize] {$H_2$} -- (5.3,7);
% p2 cannot fit
\pkgitem{6.2}{2.5}{1}{3}{orange!70}{$p_2$}
\node[red,font=\Large] at (6.7,1.4) {$\times$};
\node[font=\scriptsize] at (2.5,-1.1) {free space split: $2+1$};
\end{scope}
% ---------------- (c) after the cut: H2 = 8 ----------------
\begin{scope}[shift={(26.4,0)}]
\node[font=\small\bfseries,anchor=south west] at (-.6,9.2)
  {(c) cut: $H_2\ge 8$};
\draw[line width=.8pt] (0,0) -- (0,8.6) (5,0) -- (5,8.6);
\draw[line width=.8pt] (0,0) -- (5,0);
\pkgitem{0}{2}{1}{6}{gray!45}{1}
\pkgitem{4}{2}{1}{6}{gray!45}{2}
\pkgitem{0}{1}{3}{1}{gray!15}{3}
\pkgitem{1}{4}{3}{1}{gray!15}{4}
\pkgitem{2}{0}{3}{1}{gray!15}{5}
\pkgitem{1}{5}{1}{3}{gray!30}{6}
\pkgitem{3}{5}{1}{3}{gray!30}{7}
\pkgitem{3}{2}{1}{2}{gray!60}{8}
\pkgitem{1}{2}{1}{2}{gray!60}{9}
\pkgitem{2}{3}{1}{1}{orange!50}{$p_1$}
\pkgitem{2}{5}{1}{3}{orange!70}{$p_2$}
\draw[dashed] (-.3,8) node[left,font=\scriptsize] {$8$} -- (5.3,8);
\node[font=\scriptsize] at (2.5,-1.1)
  {$\bar t\sum_{j\in C}x_{2jp^s_j}-H_2\le \bar t(|C|-1)$};
\end{scope}
\end{tikzpicture}
\caption{The decomposition on a two-strip example:
\emph{(a)}~master view (column loads only; $H_1=3$, $H_2=7$);
\emph{(b)}~the y-check fails at $H_2=7$ although every column load
equals $7$ (hatched: largest remaining free space);
\emph{(c)}~the same configuration is feasible at $H_2=8$, so
cut~\eqref{eq:cut} with $\bar t=8$ prices it. The instance is the
reduction gadget of \citet[Fig.~2]{cote2014} with $B=2$ and
partition weights $\{1,3\}$; all coordinates verified by CP-SAT.}
\label{fig:framework}
\end{figure}

A structural difference from \citet{cote2014} deserves emphasis.
Their algorithm BLUE solves a fixed-height recognition problem
SPP($L$) and searches over the scalar $L$ externally; its cuts are
therefore height-free no-good cuts, valid because the height at
which infeasibility is detected is exactly the height being tested.
In the GMSPP no such scalar exists: the strips' heights are
interdependent decision variables, coupled through the assignment
and, under the makespan, through the objective row. BendM therefore
keeps the $H_i$ in the master, and the height must travel with the
cut itself, in the following \emph{$H_i$-aware} conditional form.
Let
$C\subseteq\mathcal{J}$ be a set of items on strip $i$ at positions
$p^s_j$ whose y-check is infeasible for target height $t$, and let
$\bar t=t+1$. The cut
\begin{equation}\label{eq:cut}
  \bar t \textstyle\sum_{j\in C} x_{ijp^s_j} - H_i
  \;\le\; \bar t\,\bigl(|C|-1\bigr)
\end{equation}
states that if all items of $C$ are placed at these positions, then
$H_i\ge \bar t$.

\begin{proposition}\label{prop:cut}
Inequality \eqref{eq:cut} is valid for
\eqref{eq:np-assign}--\eqref{eq:np-colload} under either objective
\eqref{eq:obj-area} or \eqref{eq:obj-makespan}.
\end{proposition}
\begin{proof}
If $\sum_{j\in C} x_{ijp^s_j}\le |C|-1$, the left-hand side is at most
$\bar t(|C|-1)$ since $H_i\ge 0$. Otherwise all items of $C$ are packed
at positions $p^s_j$ on strip $i$; by infeasibility of the y-check at
height $t$, any feasible packing of strip $i$ has height at least
$\bar t$, and the inequality reduces to $H_i\ge\bar t$. The argument
involves only feasibility of strip $i$ and the variable $H_i$, hence is
independent of the objective row.
\end{proof}

Intuitively, \eqref{eq:cut} is a conditional no-good cut: it does not
forbid the item--position combination outright (a taller strip may
accommodate it) but prices it at $H_i\ge\bar t$. Two strengthening
steps follow, both applied per strip. First, $C$ is reduced to a
\emph{minimal infeasible subset} (MIS)---a smallest set of items
that, kept at their positions, still admits no feasible
stacking---by greedily removing items while re-running the y-check
on the remainder.

Second, the surviving fixed positions are relaxed
to intervals by an LP adaptation of the conflict-graph lifting of
\citet[\S4.2]{cote2014}. Let
$\mathcal{K}(j)=\{k\in C\setminus\{j\}:\ j \text{ and } k \text{
share a column at the positions } p^s\}$. The intervals
$[\ell_j,r_j]\ni p^s_j$ are computed by the linear program
\begin{equation}\label{eq:lift-lp}
  \max\ \sum_{j\in C}(r_j-\ell_j)
  \quad\text{s.t.}\quad
  \ell_j + w_{ij} \ge r_k + 1
  \ \ \forall j\in C,\ k\in\mathcal{K}(j),
  \qquad 0\le \ell_j\le p^s_j\le r_j\le W_i-w_{ij},
\end{equation}
whose conflict constraints ensure that item $j$ at its
\emph{leftmost} position still overlaps item $k$ at its
\emph{rightmost} position, for every originally conflicting pair.
The LP has $2|C|$ variables and $O(|C|^2)$ constraints and solves
in negligible time.

For any choice of positions within the
intervals, every originally conflicting pair therefore still shares
a column---the conflict graph can only gain edges---so the y-check
infeasibility at height $t$ is preserved by the argument of
\citet[\S4.2]{cote2014}, which we do not restate.
Proposition~\ref{prop:cut} then applies verbatim with each term
$x_{ijp^s_j}$ in \eqref{eq:cut} replaced by
$\sum_{p\in[\ell_j,r_j]\cap\mathcal{W}_i(j)} x_{ijp}$: a single
inequality excludes every position combination within the
intervals.

Symmetry among interchangeable strips (equal width and, in the
strip-dependent case, identical dimension columns) is broken by
$H_a\ge H_b$ orderings within each group.

Algorithm~\ref{alg:bendm} summarizes BendM.

\begin{algorithm}[h!]
\caption{BendM: combinatorial Benders decomposition for the GMSPP}
\label{alg:bendm}
\begin{algorithmic}[1]
\State build the master
  \eqref{eq:np-assign}--\eqref{eq:np-colload} with the linking
  inequalities \eqref{eq:link} and symmetry-breaking orderings
\State solve the master by branch-and-cut; at each integer
  incumbent $(\hat x,\hat H)$:
\For{each strip $i$ with assigned items $C_i$ at positions $p^s$}
  \State run the y-check for $(C_i, p^s)$ at target height
    $\hat H_i$
  \If{infeasible}
    \State reduce $C_i$ to a minimal infeasible subset (greedy
      removal with y-check)
    \State lift the surviving positions to intervals via the LP
      \eqref{eq:lift-lp}
    \State add the $H_i$-aware conditional cut \eqref{eq:cut}
      (lifted form) and reject the incumbent
  \EndIf
\EndFor
\State \Return the best y-check-feasible incumbent and the master
  bound at the time limit
\end{algorithmic}
\end{algorithm}

\paragraph{Implementation note.}
Our y-check implements the niche-based enumeration of
\citet[\S3.2]{cote2014} with its bounding and dominance criteria,
in Python accelerated by Numba; its design was informed by the
original C implementation kindly shared by J.-F.~C\^ot\'e. The
1CBP-based pruning of the original is omitted, while joint
feasibility verification over lifted intervals and a MIP fallback
are added. All cross-method comparisons in this paper are internal,
so no conclusion depends on matching the speed of the original
implementation.

\section{A three-phase matheuristic}\label{sec:matheuristic}

This section presents the second solution approach of this paper: a
matheuristic that couples adaptive large neighborhood search over
item-to-strip assignments with exact optimization components---a
per-strip constraint-programming endgame and the dual-bound
portfolio---so that every run, at any instance scale, returns a
solution together with a machine-independent optimality
certificate. Its design follows a structural division of labor
(\S\ref{sec:diagnostic}): the entire search effort is spent on the
\emph{assignment} level, while the per-strip \emph{packing} level
is delegated---to a fast constructive packer inside the search, and
to an exact solver once at the end. Accordingly, the algorithm
separates the two levels: an adaptive large neighborhood search (ALNS)
explores assignments (\S\ref{sec:alns}), a local search polishes them
(\S\ref{sec:ls}), and a constraint-programming endgame re-optimizes
each strip's packing exactly under the final assignment
(\S\ref{sec:endgame}). We refer to the complete pipeline---the
three search phases together with the dual-bound portfolio---as
\emph{the package}; it outputs a solution together with a certified
optimality gap, without any branch-and-bound on the full problem.

\subsection{The two-level division of labor}\label{sec:diagnostic}

A GMSPP solution has two levels: the assignment of items to strips,
and the packing within each strip. The two are structurally
asymmetric. Given an assignment, the packing decomposes into
independent single-strip SPPs of moderate size ($10$--$50$ items
per strip on our instances)---subproblems for which strong
constructive packers exist and which an exact solver can close
outright. The assignment, by contrast, couples all strips through
the objective and, in the strip-dependent model, changes every
item's dimensions; no comparably strong local component is
available for it.

The design conclusion is a division of labor with
two regimes dictated by call frequency: \emph{inside} the search,
where packings are evaluated millions of times, each strip is
packed by a fast constructive packer; \emph{at the end}, once, each
strip's packing is re-optimized by an exact solver, which either
improves it or certifies its optimality. The search itself
therefore only ever manipulates assignments. The computational
study validates this division ex post: the packings produced by the
fast packer and local search turn out to be per-strip optimal on
every family and at every budget we test
(Section~\ref{sec:ablation} and \appref{app:zerowaste}), so
the exact endgame acts as certifier and safety net, and the
residual deficit of the package at the overall level is
attributable to the assignment search alone
(Section~\ref{sec:campaign}). This division motivates Phases~1--3
below.

\subsection{Phase 1: ALNS over assignments}\label{sec:alns}

A solution is represented by the assignment vector only; the packing
of each strip is delegated to a fast constructive packer. We maintain
per-strip height maps and evaluate insertions incrementally.

\paragraph{Per-strip packer.}
Each strip is packed by the better of (i) the best-fit skyline
heuristic of \citet{burke2004}, which repeatedly fills the lowest gap
with the widest fitting item, and (ii) bottom-left-fill under a
height-decreasing order. Within the search loop this \emph{light}
packer is used; whenever a new best solution is found, a \emph{strong}
variant additionally tries width- and area-decreasing and randomly
shuffled orders. All dimensions are those the items take on the strip
at hand, so strip-dependent instances are handled transparently.

\paragraph{Destroy and repair operators.}
The operator pool is deliberately the standard adaptive-LNS
portfolio \citep{ropke2006}, adapted to the assignment
representation. Four destroy operators remove a fraction
$\delta\sim U[\delta_{\min},\delta_{\max}]$ of the items: uniformly at
random; by largest objective contribution (with multiplicative noise);
by similarity to a random seed item (Shaw removal); or by emptying one
strip chosen with probability proportional to its contribution. Four
repair operators reinsert the removed items: greedy minimum-cost
insertion; regret-2 insertion; uniformly random feasible strip; and
round-robin spreading. Operator selection uses roulette-wheel weights
updated every $100$ iterations as in \citet{ropke2006}, and
acceptance follows simulated annealing whose temperature cools
geometrically in wall-clock time from $5\%$ to $0.01\%$ of the
initial objective value, so the schedule adapts to any budget. We
use $(\delta_{\min},\delta_{\max})=(0.10,0.40)$; the local-search
and endgame phases receive fixed budgets of $10$ and $30$ seconds.

\paragraph{Objective surrogates.}
For objective \eqref{eq:obj-area} the insertion cost of item $j$ on
strip $i$ is $W_i\Delta H_i$. For \eqref{eq:obj-makespan} this
degenerates---insertions on non-bottleneck strips have zero
cost---so we use the lexicographic surrogate
$\kappa\cdot\Delta H^{\max} + W_i \Delta H_i$ with a large constant $\kappa$,
which prefers not to raise the makespan and, among such insertions,
balances load. For the same reason, solution comparisons throughout
the search use the lexicographic key
$(H^{\max},\ \sum_i W_i H_i)$ under \eqref{eq:obj-makespan}.

\subsection{Phase 2: local search}\label{sec:ls}

The best solution found by ALNS is polished by a first-improvement
local search over two neighborhoods: relocating a single item to
another feasible strip, and swapping two items assigned to different
strips. Moves are evaluated by repacking the two affected strips with
the light packer and compared with the lexicographic key above. The
phase runs within a small time budget and applies the fine-grained
moves that the coarser destroy--repair steps miss; on large
instances it carries most of the late improvement
(Section~\ref{sec:ablation}).

\subsection{Phase 3: exact endgame}\label{sec:endgame}

Finally, the assignment is frozen and the packing of each strip is
re-optimized exactly: minimize the strip height subject to
non-overlap---i.e., the single-strip SPP on the items assigned to
that strip. The endgame is agnostic to which exact SPP method fills
this role; we use the CP-SAT solver of OR-Tools via its
two-dimensional no-overlap propagator, which is robust and
tuning-free at the subproblem sizes arising here ($10$--$50$
items).

C\^ot\'e-style machinery is not the natural fit for this
slot: the y-check alone does not suffice, since it decides
feasibility for \emph{fixed} $x$-positions whereas the endgame
re-optimizes both coordinates, and the full master--slave apparatus
of Section~\ref{sec:exact} restricted to one strip carries
model-building and cut-generation costs that a monolithic CP solver
avoids on subproblems it closes in seconds. The assignment of exact
components is thus symmetric in the call profile: inside BendM, the
y-check---invoked at every incumbent and repeatedly during MIS
reduction and lifting, in microseconds, with its fixed-position
structure feeding the cuts---is the right oracle; in the endgame,
invoked once per run on a free-coordinate optimization problem,
CP-SAT is.

Strips whose current height already equals the per-strip lower bound
$\max\{\lceil \sum_j w_{ij}h_{ij}/W_i\rceil, \max_j h_{ij}\}$ are
skipped. Under \eqref{eq:obj-makespan} strips are processed
bottleneck-first, re-selecting the tallest strip after each
improvement; under \eqref{eq:obj-area} they are processed by
decreasing height-to-bound gap. The endgame runs once, at the
end of the search, within a fixed time budget; in the parallel
setting it is applied to a few best candidate assignments
(\S\ref{sec:parallel}).

This phase is what makes the algorithm a matheuristic in the sense of
model-based metaheuristics: the neighborhood search never needs to
find good packings, only good assignments, and the exact solver is
invoked on subproblems whose size ($10$--$50$ items, one strip) is
well within its capabilities. The design mirrors the origin of large
neighborhood search in constraint programming \citep{shaw1998}.

\subsection{One algorithm, two objectives}\label{sec:touchpoints}

As with the exact framework, a single implementation serves both
objectives, because both are monotone in the strip heights and the
packing level therefore always benefits from lower strips. The
objective touches the algorithm in exactly five places: the
evaluation function; the insertion surrogate (\S\ref{sec:alns}); the
lexicographic comparison key (\S\ref{sec:ls}); the contribution
proxies of the \emph{worst} and \emph{strip} destroy operators; and
the strip-selection order of the endgame (\S\ref{sec:endgame}).
Everything else---representation, packer, operator pool, acceptance,
parallelization---is shared verbatim, so neither objective is served
by a specialized procedure.

\subsection{Parallel multi-start and certified gaps}
\label{sec:parallel}

Independent restarts with different seeds are embarrassingly parallel:
$r$ runs execute Phases~1--2 on separate cores, and the endgame is
then applied to the $K$ best solutions with \emph{distinct}
assignments (we use $r=10$, $K=3$): a slightly worse assignment may
repack strictly better, and spreading the single exact phase over a
few distinct candidates recovers such cases.
Algorithm~\ref{alg:package} summarizes the complete pipeline. The final output is the
solution value together with the certified gap to the best available
dual bound: under the makespan objective, the maximum of LP-PC, the
scheduling bound $Z^{\mathrm{ms}}$ of \S\ref{sec:mmsp}, and the
trivial height bound $h^*=\max_j\min_i h_{ij}$
(Section~\ref{sec:relax-exp}); under the total-area objective, LP-PC
alone, which dominates the elementary area bound
(\S\ref{sec:mmsp}). All members are computable in seconds and are
evaluated once in the master process. \begin{algorithm}[h!]
\caption{The certified matheuristic package}
\label{alg:package}
\begin{algorithmic}[1]
\State compute the dual bound $L \gets \max(\text{LP-PC},\,
  Z^{\mathrm{ms}},\, h^*)$ \Comment{portfolio, \S\ref{sec:relax-exp}; members as applicable}
\For{$k = 1,\dots,r$ \textbf{in parallel}}
  \State \textbf{Phase 1 (ALNS):} destroy--repair over assignments
    with SA acceptance until the phase budget elapses
  \State \textbf{Phase 2 (LS):} first-improvement relocate/swap on
    the run's best solution
\EndFor
\State select the $K$ best solutions with distinct assignments
\For{each selected solution}
  \State \textbf{Phase 3 (endgame):} per strip, re-optimize the
    packing with CP-SAT (skip strips at their per-strip bound)
\EndFor
\State \Return best solution $z$ and certified gap
  $(z - \lceil L\rceil)/z$
\end{algorithmic}
\end{algorithm}

We emphasize that the certificate
is machine-independent: it compares two numbers---a feasible value and
a relaxation optimum---and involves no time-limited branch-and-bound.

\section{Computational study}\label{sec:experiments}

% ================= 6.1 Setup =================
\subsection{Setup and fairness protocol}\label{sec:setup}

All algorithms are implemented in Python~3.11; the computationally
intensive kernels run in native code (Gurobi~13.0 for linear and
integer programs, OR-Tools CP-SAT for packing subproblems, NumPy and
Numba for the packing heuristics and the y-check). Experiments were
run under Windows on an Intel Core i9-13900F (3.10\,GHz, 24 cores)
with 64\,GB of RAM. Every Gurobi solve (the BigM and BigM-LE MIPs, the
BendM master, and the scheduling bound \eqref{eq:mmsp}) is capped at
16 threads; the package runs 10 single-threaded search
workers in parallel and its CP-SAT endgame may use all cores within
its 30-second budget.

Comparisons follow a three-tier protocol designed so that no claim
depends on hardware or time-budget asymmetries.
\emph{Tier 1 (relaxation strength):} lower bounds are compared as
optimal values of their relaxations---machine- and budget-independent
quantities. \emph{Tier 2 (certified gaps):} the matheuristic package
reports the gap between its solution and the best dual bound, again a
comparison of two numbers free of branch-and-bound budgets.
\emph{Tier 3 (equal-budget comparison):} all methods in a table
receive the same wall-clock budget $T=1200$~s and thread cap,
adopting the per-instance budget of \citet{cote2014}. We make no
cross-implementation speed claims. Since all instance data are
integer, both objectives take integer values; in Tier~2 every dual
bound is therefore rounded up to the nearest integer before gaps are
computed, uniformly across all bounds (Tier~1 reports the raw
relaxation optima, matching the definitions of the compared bounds).

\paragraph{Comparison and reimplementation policy.}
We compare against prior methods through their \emph{published}
results on reproducible instances, and we reimplement only components
whose outcome is independent of implementation quality. Concretely,
the scheduling bound \eqref{eq:mmsp} of \citet{vasilyev2023} is an
integer program whose optimal value does not depend on who implements
it, so we recompute it on every instance; their skyline best-fit
heuristic with randomized local search, by contrast, is an engineered
search procedure whose performance depends on scoring, restart, and
stopping details, and a reimplementation would be open to objections
from both directions.

We note that the skyline idea itself is not
absent from our method: the per-strip packer of
Section~\ref{sec:alns} follows the best-fit skyline principle of
\citet{burke2004}, which \citet{vasilyev2023} also build on; the two
approaches differ at the assignment-search level, where we use ALNS
with an exact endgame instead of randomized swaps. We also note that
their skyline procedure targets the makespan alone, whereas every
method in this study handles both objectives. Finally, our
quality claims for the matheuristic do not rest on
heuristic-to-heuristic comparison at all: they are certificates against dual
bounds (Tier~2), which require no external baseline.

% ================= 6.2 Instances =================
\subsection{Instances}\label{sec:instances}

Table~\ref{tab:instances} summarizes the five instance families and
their roles; the remainder of this subsection gives the rationale
for each. All tables in this section report aggregates; the
per-instance result files (CSV) underlying every table are included
in the code and data release.

\begin{table}[h!]
\centering\small
\caption{Instance families at a glance.}
\label{tab:instances}
\begin{tabular}{llrrll}
\toprule
Family & Dimensions & $n$ & $m$ & Objectives & Role\\
\midrule
CLASS & fixed & 20--100 & 2--3 & both &
  primary (Tiers 1--3)\\
Type I/II & strip-dep. & 150--800 & 2--4 & makespan &
  scheduling lens (Tiers 1--3)\\
ZDF splits & strip-dep. & 145--766 & 2--4 & makespan &
  scheduling lens (Tiers 1--3)\\
ZDF stress & strip-dep. & 1258--2532 & 2--8 & makespan &
  scale (Tiers 1--3)\\
zero-waste & fixed & 17--197 & 2--3 & both &
  calibration only (\appref{app:zerowaste})\\
\bottomrule
\end{tabular}
\end{table}

Our primary fixed-dimension family derives from the CLASS collection
of Berkey--Wang and Martello--Vigo (500 SPP instances in ten classes
with distinct item-shape distributions), the main test bed of
\citet{cote2014}. We use classes 1--10 with
$n\in\{20,60,100\}$, two instances per group, converted to GMSPP
under both objectives by the following deterministic strip
protocol: an SPP instance of width $W$ becomes
$m=2$ strips of widths $\{W,\,\mathrm{round}(1.2W)\}$ or $m=3$
strips of widths
$\{\mathrm{round}(0.8W),\,W,\,\mathrm{round}(1.2W)\}$. The widths
straddle the original $W$, which creates genuine width heterogeneity
while guaranteeing feasibility (every item fits the original width)
and keeping the conversion reproducible without random draws.

We deliberately exclude the zero-waste instances of
\citet{burke2004} from the main comparisons. As \citet[p.~657]{cote2014}
already observed for the single-strip case, these instances are
``built ad hoc to have zero waste,'' so that ``the computation of
sophisticated lower bounds is useless''; in our preliminary runs every
relaxation of the lattice \eqref{eq:lattice} collapsed to the area
bound on them (identical values on 57 of 70 conversions). The
multi-strip setting adds a second degeneration: under the total-area
objective, the optimal solution packs all items into the original
strip with zero waste and leaves the remaining strips empty, so the
assignment dimension is vacuous at the optimum. Zero-waste instances
remain an interesting stress test for primal heuristics---there the
area bound \emph{equals} the optimum, so certified gaps measure true
suboptimality exactly---and we report them in that role only.
% (zero-waste stress test: Appendix A)

The scheduling lens is covered by regenerated Type~I and Type~II
instances following the generation protocol of \citet{vasilyev2023}
and by splits of the ZDF collection according to their CX/ZDF
protocol; since neither the strip widths nor the shuffle seeds of
their instances are specified in the paper, we document our protocol
and release all instances. Strip widths are drawn uniformly at
random (with fixed seeds) for Type~I/II, matching their described
procedure; for the ZDF splits we use the original width for every
strip, the neutral choice given that theirs are unspecified---there
the heterogeneity enters entirely through the strip-dependent
dimensions, making these instances the purest expression of the
scheduling lens of Section~\ref{sec:lenses}.

Their published lower bounds are not
reproducible from instance data alone for a second reason: as noted
in \S\ref{sec:mmsp}, they incorporate variable fixings derived from
their heuristic's upper bound, whereas we recompute the pure
relaxation \eqref{eq:mmsp}. We contacted the authors for the
original instance files; as these were not available by the time of
writing, all comparisons use the regenerated
instances above, which we release in full.

% ================= 6.3 Tier 1: relaxation strength =================
\subsection{Relaxation strength}\label{sec:relax-exp}

Table~\ref{tab:relax} reports, for every makespan instance, the four
members of the lattice \eqref{eq:lattice}: the big-$M$ LP, the
liquified LP base, Vasilyev's scheduling bound $Z^{\mathrm{ms}}$, and
LP-PC, together with the trivial combinatorial bound
$h^*=\max_j\min_i h_{ij}$ (every item forces the height of the strip
that receives it, so the makespan is at least $\min_i h_{ij}$ for
each $j$; the bound is the strip-dependent generalization of the
maximum-item-height bound classical in strip packing).

The table is restricted to the makespan because that is the only
objective on which the bound ranking is genuinely open: under total
area, the big-$M$ LP and the liquified analogue are dominated by
Propositions~\ref{prop:dominance} and~\ref{prop:mmsp-lp}, and the
one candidate not dominated in general---the trivial cost bound
$\max_j\min_i C_iW_ih_{ij}$, the total-area analogue of
$h^*$---turns out to be strictly weaker than LP-PC on all $120$
total-area CLASS instances (splitting an item across strips dilutes
a \emph{maximum} of heights but not a \emph{sum} of
width-weighted heights, so the mechanism that lets $h^*$ win under
the makespan has no total-area counterpart on strips of similar
widths). LP-PC is the unique best total-area bound throughout the
study.

\begin{table}[h!]
\centering\small
\caption{Tier 1 --- relaxation strength on all makespan instances.
Columns 3--7: mean percentage below the best bound of the portfolio
(0.0 = this bound is the best on every instance of the row).
Last column: largest gain of LP-PC over $Z^{\mathrm{ms}}$ in the row.
LP-PC is not computed beyond $n=500$ (stress rows); the big-$M$ LP is
not computed beyond $n=120$.}
\label{tab:relax}
\begin{tabular}{lrrrrrrr}
\toprule
 & & \multicolumn{5}{c}{mean \% below best bound} &
   max gain (\%)\\
\cmidrule(lr){3-7}
Family & \# & big-$M$ LP & liq.\ LP & $Z^{\mathrm{ms}}$ & LP-PC &
$h^*$ & LP-PC vs $Z^{\mathrm{ms}}$\\
\midrule
CLASS 1 & 12 & 89.0 & 0.0 & 0.0 & 0.0 & 0.0 & 0.1 \\
CLASS 2 & 12 & 76.9 & 9.7 & 9.7 & 9.7 & 0.0 & 0.0 \\
CLASS 3 & 12 & 86.4 & 1.2 & 1.2 & 0.0 & 1.2 & 7.3 \\
CLASS 4 & 12 & 76.0 & 11.1 & 11.1 & 11.1 & 0.0 & 0.0 \\
CLASS 5 & 12 & 87.5 & 2.4 & 2.4 & 0.0 & 2.4 & 5.9 \\
CLASS 6 & 12 & 75.2 & 13.0 & 13.0 & 13.0 & 0.0 & 0.0 \\
CLASS 7 & 12 & 88.9 & 10.7 & 10.7 & 0.0 & 10.7 & 32.9 \\
CLASS 8 & 12 & 88.4 & 0.1 & 0.1 & 0.0 & 0.1 & 0.6 \\
CLASS 9 & 12 & 94.1 & 11.1 & 11.1 & 0.0 & 11.1 & 28.0 \\
CLASS 10 & 12 & 82.4 & 2.1 & 2.1 & 0.1 & 1.8 & 12.7 \\
T1 & 6 & -- & 0.0 & 0.0 & 0.0 & 0.0 & 0.0 \\
T2 & 6 & -- & 0.5 & 0.5 & 0.0 & 3.4 & 2.0 \\
ZDF & 6 & -- & 31.2 & 31.1 & 17.4 & 0.2 & 57.8 \\
ZDFS & 4 & -- & 90.1 & 90.1 & -- & 0.0 & -- \\
\bottomrule
\end{tabular}
\end{table}

Four observations stand out (Table~\ref{tab:relax}). First, the
proven relations hold throughout: on every instance the big-$M$ LP
never exceeds LP-PC (Proposition~\ref{prop:dominance}) and the
liquified LP never exceeds LP-PC (Proposition~\ref{prop:mmsp-lp}),
with zero violations; the big-$M$ LP trails the best bound by
$75$--$94\%$ on every class, the computational face of its missing
area accumulation.

Second, the comparison between the two incomparable refinements is
strongly asymmetric. On CLASS, LP-PC is strictly stronger than
$Z^{\mathrm{ms}}$ on 52 of 120 instances, equal on 50, and weaker on
18---and where $Z^{\mathrm{ms}}$ wins, its advantage never exceeds
$0.05\%$, whereas LP-PC's advantage reaches $32.9\%$ (class~7) and
$57.8\%$ (ZDF-derived). The integrality increment
$Z^{\mathrm{ms}}-Z^{\mathrm{ms}}_{\mathrm{LP}}$ is below $0.05\%$
everywhere: essentially the entire difference between the bounds is
\emph{geometry}, not integrality.

Third, the families organize into three regimes visible in
Table~\ref{tab:relax}: area-dominated classes (1, 8, and the Type~I
instances), where all portfolio members essentially coincide;
geometry classes (3, 5, 7, 9, 10 and Type~II), where LP-PC is the
unique best bound; and wide-strip classes (2, 4, 6) together with
the ZDF splits, where the trivial height bound $h^*$ dominates
\emph{both} LP-based bounds---by $17$--$26\%$ on average on classes
2, 4, 6---because the strip is wide relative to the items, the area
bound is weak, and a tallest item pins the makespan. On the stress
splits (last row) this regime is extreme: $Z^{\mathrm{ms}}$ trails
$h^*$ by $90\%$, so the only previously available bound would report
gaps of that order there.

Fourth, following the bound-portfolio practice standard in strip
packing \citep[\S5.1]{cote2014}, our dual bound is the maximum of the
three surviving members---LP-PC, $Z^{\mathrm{ms}}$, and $h^*$;
the dominance results of Section~\ref{sec:exact} guarantee that the
big-$M$ LP and the liquified LP need never be computed, and each
regime of the taxonomy above has one portfolio member that covers
it. All three bounds are computable in seconds ($h^*$ in $O(nm)$).

% ================= 6.4 Tier 2 =================
\subsection{Certified gaps of the matheuristic package}
\label{sec:sweep-exp}

Tables~\ref{tab:sweep-class} and~\ref{tab:sweep-sched} report the
certified gaps of the package under a $300$-second budget on every
instance of the study.

\begin{table}[h!]
\centering\small
\caption{Tier 2 --- certified optimality gaps of the matheuristic
package on the CLASS families ($300$\,s per instance, $10$ parallel
runs with distinct seeds; $12$ instances per cell).  Certified
gaps (in \%) are against the rounded-up bound portfolio; CV is the
coefficient of variation---std/mean, in \%---of the ten runs'
final objective values, averaged over the cell.}
\label{tab:sweep-class}
\begin{tabular}{lrrrrrrrr}
\toprule
 & \multicolumn{4}{c}{total area} & \multicolumn{4}{c}{makespan}\\
\cmidrule(lr){2-5}\cmidrule(lr){6-9}
Class & mean\,\% & max\,\% & \#opt & CV\,\% & mean\,\% & max\,\% & \#opt & CV\,\%\\
\midrule
CLASS 1 & 1.1 & 2.2 & 0 & 0.43 & 2.1 & 6.2 & 2 & 0.35 \\
CLASS 2 & 2.6 & 7.5 & 0 & 0.35 & 2.6 & 9.1 & 4 & 0.28 \\
CLASS 3 & 3.7 & 6.0 & 0 & 0.50 & 3.8 & 10.3 & 0 & 0.51 \\
CLASS 4 & 4.6 & 6.6 & 0 & 0.39 & 4.3 & 7.7 & 3 & 0.33 \\
CLASS 5 & 3.2 & 5.3 & 0 & 0.44 & 4.2 & 7.3 & 0 & 0.45 \\
CLASS 6 & 5.8 & 10.1 & 0 & 0.38 & 4.7 & 8.7 & 4 & 0.34 \\
CLASS 7 & 1.7 & 5.8 & 2 & 0.27 & 1.6 & 7.8 & 5 & 0.18 \\
CLASS 8 & 4.1 & 5.6 & 0 & 0.34 & 6.5 & 10.6 & 0 & 0.49 \\
CLASS 9 & 0.9 & 1.9 & 0 & 0.63 & 0.5 & 3.0 & 5 & 0.14 \\
CLASS 10 & 3.1 & 5.5 & 0 & 0.45 & 4.4 & 6.8 & 0 & 0.37 \\
\bottomrule
\end{tabular}
\end{table}

\begin{table}[h!]
\centering\small
\caption{Tier 2 --- certified gaps on the strip-dependent
(scheduling-lens) instances, $300$\,s each.  Starred rows are the
stress instances (LP-PC skipped; portfolio $= Z^{\mathrm{ms}},
h^*$).  A gap of $0.00$ is a proven optimum; CV as in
Table~\ref{tab:sweep-class}.}
\label{tab:sweep-sched}
\begin{tabular}{lrrrrrr}
\toprule
Instance & $n$ & $m$ & obj & $\lceil\mathrm{LB}\rceil$ &
gap\,\% & CV\,\%\\
\midrule
zdf1\_4 & 145 & 4 & 15 & 15 & 0.00 & 0.00 \\
T1\_150\_2 & 150 & 2 & 658 & 621 & 5.62 & 1.06 \\
T2\_150\_3 & 150 & 3 & 907 & 845 & 6.84 & 1.69 \\
T1\_150\_3 & 150 & 3 & 492 & 462 & 6.10 & 0.92 \\
T2\_150\_2 & 150 & 2 & 1747 & 1687 & 3.43 & 0.86 \\
T2\_250\_4 & 250 & 4 & 923 & 826 & 10.51 & 1.78 \\
T1\_250\_4 & 250 & 4 & 659 & 626 & 5.01 & 0.72 \\
zdf1\_2 & 290 & 2 & 29 & 27 & 6.90 & 2.17 \\
zdf2\_2 & 330 & 2 & 32 & 29 & 9.38 & 3.06 \\
zdf3\_2 & 370 & 2 & 49 & 49 & 0.00 & 0.00 \\
T2\_400\_4 & 400 & 4 & 1426 & 1234 & 13.46 & 1.82 \\
T1\_400\_4 & 400 & 4 & 1036 & 988 & 4.63 & 0.43 \\
zdf4\_2 & 410 & 2 & 38 & 36 & 5.26 & 3.68 \\
T1\_600\_4 & 600 & 4 & 1560 & 1495 & 4.17 & 0.24 \\
T2\_600\_4 & 600 & 4 & 2023 & 1696 & 16.16 & 1.74 \\
zdf6\_2 & 766 & 2 & 74 & 74 & 0.00 & 0.00 \\
T2\_800\_4 & 800 & 4 & 2617 & 2227 & 14.90 & 1.79 \\
T1\_800\_4 & 800 & 4 & 2056 & 1977 & 3.84 & 0.67 \\
zdf12\_8$^*$ & 1258 & 8 & 9 & 9 & 0.00 & 0.00 \\
zdf8\_2$^*$ & 1266 & 2 & 100 & 100 & 0.00 & 0.00 \\
zdf10\_4$^*$ & 1266 & 4 & 31 & 31 & 0.00 & 0.00 \\
zdf10\_2$^*$ & 2532 & 2 & 159 & 159 & 0.00 & 0.00 \\
\bottomrule
\end{tabular}
\end{table}

Three findings. First, quality is uniform
and reproducible: the mean certified gap is below $6.5\%$ on every
CLASS class under both objectives (overall means $3.1\%$ for total
area and $3.5\%$ for makespan), 32 of 262 instances are solved to
proven optimality, and the coefficient of variation across the ten
independent runs has median $0.36\%$---the certificate, not seed
luck, is the binding uncertainty.

Second, the certified gap inherits the regime structure of
Section~\ref{sec:relax-exp}: the largest gaps (classes 6 and 8, the
larger Type~II instances) occur where the dual portfolio is weakest,
not where the primal search degrades. The instances that BendM
closes in the equal-budget campaign of Section~\ref{sec:campaign}
make this decomposition observable, since there the optimum is
known. On the class-8 makespan instances it closes ($3$ of $12$),
the package's certified gap of $8.0\%$ resolves into a true gap of
$0.8\%$ and $7.2$ percentage points of dual slack; classes 1, 3,
and 5 repeat the pattern ($14$ closed instances in total, certified
gaps of $4.8$--$5.3\%$ against true gaps below $1\%$). Class~8 is
the sharpest case because it is also the family on which the whole
bound portfolio collapses to the area bound
(Table~\ref{tab:relax}): where that happens it is the bound, not
the search, that sets the certificate.

Third---the scale result---the four stress instances
(Table~\ref{tab:sweep-sched}, starred), with $n$ up to $2532$ and
$m$ up to $8$, are all solved to \emph{proven optimality} within the
same $300$-second budget: the package attains the height bound
$h^*$ exactly. We emphasize what makes this notable: these are
splits of Vasilyev et al.'s own source collection in the tall-item
regime, where the only previously available dual bound, in the pure
data-only form of \S\ref{sec:mmsp}, trails the optimum by $90\%$
(Table~\ref{tab:relax}); certificates there are
made possible by the portfolio, and the primal search reaches them.
We note candidly that $h^*$ being tight is a favorable feature of
this family; the point is that the certified framework detects and
exploits it automatically at a scale no exact method reaches.

% ================= 6.5 Tier 3 =================
\subsection{Equal-budget campaign: hierarchy and frontier}
\label{sec:campaign}

Table~\ref{tab:hierarchy} is an \emph{elimination} experiment: its
question is whether the formulation-level methods are competitive
at all, so it runs all four methods on a fixed subset chosen ex
ante to span the shape regimes---classes 1 and 3 (Berkey--Wang)
and 8 and 10 (Martello--Vigo), covering both the area-dominated
and the geometry regime of Table~\ref{tab:relax}---at the two
sizes where the answer could conceivably be yes. The verdict does
not depend on the subset, because the collapse it documents is
structural rather than class-specific; the surviving methods are
then compared exhaustively, on every instance of the study, in
Table~\ref{tab:frontier}. The plain big-$M$ formulation collapses
exactly as Proposition~\ref{prop:dominance} predicts, confirming
quantitatively the assessment of \citet{vasilyev2023} that it is
tractable only at very small sizes: a $22\%$ mean terminal
gap already at $n=20$, degrading to $82\%$ at $n=60$, with not a
single instance closed at $n=60$. BigM-LE repairs the \emph{bound}
(terminal gaps drop to $3.6\%$ and $9.2\%$) but closes no additional
instances: injecting LP-PC substitutes for the missing relaxation
strength but cannot substitute for branching structure. BendM
dominates the table---$24/32$ optima at $n=20$ at a mean time of
$364$\,s---and is the only method whose gap is routinely zero.

\begin{table}[h!]
\centering\small
\caption{Tier 3 --- elimination experiment: the method hierarchy
at small size (classes
1, 3, 8, 10; both objectives; $32$ instances per row; $T=1200$\,s
each).  The terminal gap decomposes into its two sides: the primal
pair compares each method's incumbent, and the dual pair its
terminal lower bound, with the best of the four methods; \#best
counts instances at that best, $\Delta$ is the mean relative
deviation from it.}
\label{tab:hierarchy}
\begin{tabular}{llrrrrrrr}
\toprule
 & & & & \multicolumn{2}{c}{primal (incumbent)} &
\multicolumn{2}{c}{dual (bound)} & \\
\cmidrule(lr){5-6}\cmidrule(lr){7-8}
$n$ & Method & opt & gap\,\% & \#best & $\Delta$\,\% &
\#best & $\Delta$\,\% & time (s)\\
\midrule
20 & BigM & 5/32 & 22.30 & 12/32 & 0.87 & 5/32 & 21.34 & 1017 \\
 & BigM-LE & 5/32 & 3.57 & 15/32 & 1.14 & 8/32 & 1.81 & 1024 \\
 & BendM & 24/32 & 0.69 & 31/32 & 0.00 & 32/32 & 0.00 & 364 \\
 & Package & 0/32 & 4.49 & 12/32 & 0.89 & 3/32 & 2.96 & 1101 \\
\midrule
60 & BigM & 0/32 & 81.66 & 0/32 & 11.21 & 0/32 & 79.38 & 1202 \\
 & BigM-LE & 0/32 & 9.15 & 0/32 & 8.80 & 16/32 & 0.07 & 1202 \\
 & BendM & 6/32 & 1.12 & 32/32 & 0.00 & 31/32 & 0.00 & 978 \\
 & Package & 0/32 & 3.42 & 2/32 & 2.34 & 17/32 & 0.07 & 1102 \\
\bottomrule
\end{tabular}
\end{table}

Two reading notes. First, the primal/dual columns decompose each
terminal gap into its two sides, and the decomposition is
instructive. BigM-LE's repair is purely dual: it lifts the bound
side from $21.3\%$ to $1.8\%$ below the best bound (to near-parity
at $n=60$) while its incumbents are no better---at $n=20$ slightly
worse---than plain big-$M$'s: certification, not solutions, in its
sharpest form (cf.\ \S\ref{sec:anatomy}). The package's binding
side flips with size: at $n=20$ its incumbents are near-best
($0.89\%$) and its gap is mostly bound slack ($2.96\%$), whereas at
$n=60$ its portfolio bound is essentially tied with BendM's
($0.07\%$) and the gap has become primal ($2.34\%$). BendM leads
on both sides throughout. Second,
the package's zero optima in Table~\ref{tab:hierarchy} are an
artifact of the class subset: its proven optima concentrate in
classes 2, 4, 6, and 7---where $h^*$ or LP-PC is tight
(Table~\ref{tab:relax})---none of which belong to the hierarchy
subset $\{1,3,8,10\}$; over all ten classes it proves $13$ optima
at $n=20$ (Table~\ref{tab:frontier}).

The hierarchy is the computational argument for what we call an
exact method here: the contribution is not the model but the
decomposition machinery around it.

\begin{table}[h!]
\centering\small
\caption{Tier 3 --- the exact/matheuristic frontier: BendM vs the
package under a common budget of $1200$\,s.  ``Duel'' counts
instances whose best solution is strictly better under
BendM\,/\,tied\,/\,strictly better under the package.  On the stress
band BendM returned no feasible solution on three of the four
instances and no valid lower bound on any (gap reported as $100$);
building its master alone exceeded the budget several-fold.}
\label{tab:frontier}
\begin{tabular}{lrrrrr}
\toprule
 & \multicolumn{2}{c}{BendM} & \multicolumn{2}{c}{Package} & \\
\cmidrule(lr){2-3}\cmidrule(lr){4-5}
Band & opt & gap\,\% & opt & gap\,\% & Duel (B/t/P)\\
\midrule
CLASS $n=20$ & 58/80 & 1.02 & 13/80 & 3.52 & 37/43/0 \\
CLASS $n=60$ & 16/80 & 1.66 & 1/80 & 3.23 & 60/17/3 \\
CLASS $n=100$ & 13/80 & 1.68 & 4/80 & 2.97 & 58/15/7 \\
sched.\ $n=145$--$800$ & 2/18 & 2.93 & 3/18 & 6.90 & 15/3/0 \\
stress $n=1258$--$2532$ & 0/4 & 100.00 & 4/4 & 0.00 & 0/0/4 \\
\bottomrule
\end{tabular}
\end{table}

Table~\ref{tab:frontier} extends the comparison to the full range.
The verdict is unambiguous, and it was not the one we anticipated
when designing the study: \emph{BendM remains the strongest method
throughout the range in which its master can be built}, winning the
primal comparison on $170$ of the $258$ instances up to $n=800$
($78$ ties, $10$ package wins) and delivering the tighter certified
gap in every band, with terminal gaps of only $1.1$--$1.6\%$ on the
Type~I/II instances as large as $n=800$. The package tracks it
within $1.3$--$2.5$ percentage points of certified gap on CLASS and
matches its incumbent on roughly half of the small instances---but
where BendM proves optimality, the package can typically only
certify to within the LP-PC slack. The primal deficit has a clean
attribution: since the endgame certifies that the package's
per-strip packings are optimal for its final assignment
(Section~\ref{sec:ablation}), whatever separates the two
incumbents lies at the \emph{assignment} level---at these sizes,
BendM's branch-and-cut explores assignments more effectively than
neighborhood search does.

One mid-size split (zdf6,
$n=766$) illustrates the converse: both methods find the same
solution, but only the package proves it optimal, because its
certificate draws on the bound portfolio ($h^*$ is tight there)
while BendM's internal bound does not.
The frontier itself is located by the stress band: at
$n\ge 1258$, BendM produced no feasible solution on three of the
four instances and no valid lower bound on any---building the master
alone consumed several times the budget---while the package proved
optimality on all four. The hand-off is governed less by $n$ alone
than by the size of the master---roughly the number of normal
positions times columns, which grows with both $n$ and the strip
widths. The generator instances at $n=800$ have widths below $150$
and remain well within reach; the stress splits combine
$n\ge 1258$ with widths of $3000$--$6000$, inflating the master by
two orders of magnitude rather than by the factor $1.6$ that item
count alone would suggest. The mid-size split zdf6 ($n=766$,
$W=3000$) sits visibly at the edge: BendM still builds its master
and matches the package's solution, but its bound already lags.
The exact/matheuristic hand-off thus occurs where the master
outgrows the budget---a frontier in model size, which our
instances happen to cross between $n=800$ and $n=1258$.

Two refinements temper the verdict. By objective, BendM's advantage
concentrates on total area, where its branch-and-cut closes the
dual side almost completely (mean terminal gap $1.1$--$1.4\%$,
$21/40$ optima at $n=20$); under the makespan the two methods are
already near parity at $n=100$ ($2.3\%$ vs $3.4\%$ mean certified
gap), consistent with the plateau structure of min--max objectives
that weakens incumbent pruning. And by scale, the package is the
only method of the study that operates at the stress sizes at all,
where it delivers proven optima (Section~\ref{sec:sweep-exp}).
The two contributions are thus complementary in a precise,
measured sense: the rebuilt Benders machinery is the method of
choice wherever it runs; the certified package extends quality
guarantees beyond that frontier.

% ================= 6.6 Component anatomy =================
\subsection{Reading the comparison component by component}
\label{sec:anatomy}

The four methods sample the design space of GMSPP algorithms along
two axes---how the assignment level is searched, and how the
packing level is handled---summarized in Table~\ref{tab:anatomy}.
Two lessons emerge when the results of
Tables~\ref{tab:hierarchy}--\ref{tab:frontier} are read against
this anatomy.

\begin{table}[h!]
\centering\small
\caption{How each method decomposes the GMSPP, and where its
implementation effort goes.}
\label{tab:anatomy}
\begin{tabular}{lllll}
\toprule
Method & Assignment level & Packing level & Dual bound & Effort\\
\midrule
BigM & \multicolumn{2}{c}{\multirow{2}{*}{monolithic MIP, no level
  separation}} & own LP (degenerate) & minimal\\
BigM-LE & \multicolumn{2}{c}{} & injected LP-PC & small\\
BendM & master B\&C & split: $x$ in master, & master bound &
  substantial\\
 & (dual-guided) & $y$-check at fixed $x$ ($\mu$s) & ($\ge$ LP-PC)
 & \\
Package & ALNS + LS & joint $(x,y)$: fast packer & portfolio &
  substantial\\
 & (evaluation-guided) & in-loop, CP-SAT once & (LP-PC, $Z^{\mathrm{ms}}$, $h^*$) & \\
\bottomrule
\end{tabular}
\end{table}

First, \emph{performance follows structural effort, not
formulation effort}. The two formulation-level methods differ from
the two engineered methods by orders of magnitude in
implementation labor, and the gap in results matches; but BigM-LE
is the instructive control: injecting the strong bound repairs the
big-$M$ model's terminal gap while leaving its ability to close
instances unchanged, so effort invested at the formulation level
buys certification, not solutions. The effort that pays is
structural---a decomposition, or a search architecture.

Second, between the two engineered methods the effort is
comparable and the differences are purely architectural, which
localizes the verdict of Table~\ref{tab:frontier}. BendM splits
the packing level ($x$ chosen by the master, $y$ delegated to a
microsecond decision oracle) and searches assignments with
\emph{dual-guided} branch-and-cut: the master optimizes assignments
directly against the LP-PC geometry, and every cut permanently
adds global infeasibility knowledge. The package treats each
strip's packing jointly and more exhaustively per call---but
invoked far less often---and searches assignments by
\emph{evaluation-guided} neighborhood moves, with no dual
information. In hindsight the outcome is natural: the packing
level, which the package handles at least as well
(Section~\ref{sec:ablation}), was never where the game was decided;
the assignment level was, and there dual guidance beats objective
evaluations wherever the master can be built.

The package's compensating asset is structural: none of its
components ever materializes a model of the whole instance. The
packer and the neighborhood moves operate on incremental per-strip
data structures, the endgame sees one strip at a time, and the
bounds that remain available at scale ($Z^{\mathrm{ms}}$ and
$h^*$) are a small MIP and an $O(nm)$ scan. Every model-based
method, by contrast, must pay for a position-indexed model of the
entire instance---hundreds of thousands of binaries with dense
column-load constraints at the stress sizes---\emph{before}
returning anything. This is why the frontier of
Table~\ref{tab:frontier} is a cliff rather than a slope: BendM is
the best method in the study at $n=800$ (strip widths $\le 150$)
and returns neither a solution nor a bound on the stress splits
($n\ge 1258$, widths $3000$--$6000$), whose masters are two orders
of magnitude larger. Once an instance is too large to
model, search with certificates is not merely competitive with
model-based optimization; it is the only approach that produces
any quality statement at all.

% ================= 6.7 Phase analysis =================
\subsection{Where the package spends its power}
\label{sec:ablation}

Three diagnostics justify the internal design of the package and its
budget allocation. First, the exact endgame improved no strip on
\emph{any} instance of the study---CLASS, scheduling, stress, and
the zero-waste family alike, at budgets of both $60$ and $300$
seconds (\appref{app:zerowaste}): the per-strip packings
produced by the skyline packer and local search are already optimal,
and CP-SAT \emph{proves} this in $0$--$3$ seconds on typical strips.
The endgame's measured role is therefore certification and
insurance---it guarantees, at a constant cost of $30$\,s, that the
final packing quality is independent of the heuristic packer---and
this is itself a finding: a strong skyline packer plus local search
closes the packing level heuristically.

Second, the local-search
phase carries the late
improvements at scale: on instances with $n>100$ it contributes a
mean $3.0\%$ (up to $6.5\%$) improvement over the best ALNS
solution, against $0.4\%$ below $n=100$. Third, because the endgame
saturates, the time allocation gives it a constant small budget
($30$\,s) and assigns all remaining time to the assignment-level
search, whose simulated-annealing schedule cools geometrically in
wall-clock time and therefore adapts to any budget without
retuning.

\section{Conclusion}\label{sec:conclusion}

We set out to bridge two literatures: the mature exact toolbox of
single-strip packing, and the heuristic-only state of the art for
multiple heterogeneous strips with strip-dependent item dimensions.
The bridge required genuine reconstruction rather than translation:
a variable-height master with $H_i$-aware conditional Benders cuts
(proven valid), a packing-aware LP bound with height-linking
inequalities (proven to dominate both previously available LP
relaxations), and a bound portfolio whose three members provably and
empirically cover complementary instance regimes.

The comparative study delivers a map we believe will be useful to
anyone approaching the problem next. The plain big-$M$ formulation,
stated but never evaluated before, is not a viable exact method: it
collapses at $n=20$ and is off by $80\%$ at $n=60$; injecting the
LP-PC bound repairs its terminal gap but not its ability to close
instances. The rebuilt Benders decomposition is the method of
choice everywhere its master can be built---it wins the primal and
the certified comparison on every family up to $n=800$, under both
objectives, and supplies the first proven optima for the GMSPP.

Its advantage is largest under the total-area objective and
narrows to near parity under the makespan, whose min--max plateau
blunts incumbent pruning. Beyond the exact frontier, the certified
matheuristic package is the only method standing, and the bound
portfolio keeps it honest: at $n$ up to $2532$ and $m=8$ it delivers
proven optima on splits of the very collection on which the
previously available bound is off by $90\%$. The two contributions
are complementary in a measured, not rhetorical, sense.

Three questions remain open, and together they mark the boundary of
what this paper claims. On the dual side, the slack that survives is
concentrated rather than diffuse: Section~\ref{sec:sweep-exp} places
the largest certified gaps on the narrow-tall CLASS family (class~8)
and on the larger Type~II instances, and shows that on the former it
is the bound, not the search, that sets the certificate. The
strengthenings the single-strip literature would suggest do not
reach that regime. Dual-feasible functions and conservative scales
\citep{fekete2004,carlier2007} act on the ratio of item width to
container width, and on a family of narrow items that ratio stays
below one half almost everywhere, which leaves such a scale equal to
the identity; their extension to multiple heterogeneous strips is
open in any case, since the ratio is itself strip-dependent and the
resulting per-strip bounds have to be recombined through the
assignment. The binding question is instead on the height side,
where the linking inequalities \eqref{eq:link} already act---but one
item at a time. Certifying the joint height of a \emph{subset} of
items, complementing the column loads rather than duplicating them,
is the natural next step. On the primal side,
Section~\ref{sec:campaign} attributes the package's entire remaining
deficit to the assignment level; exact components acting at that
level---exact reinsertion of a destroyed assignment, or compound
relocation moves such as ejection chains---lay outside the scope of
a study designed to compare architectures rather than to tune one,
and \appref{app:zerowaste} indicates what they would have to
overcome: on jigsaw-like instances the search stays $3$--$8\%$ from
the optimum even though the endgame certifies each strip
individually. On the modeling side, orthogonal rotations and
guillotine restrictions, both excluded here in line with the
oriented, free-cutting setting of \citet{cote2014} and
\citet{vasilyev2023}, are natural extensions of the model.

\section*{Acknowledgments}

The authors thank Jean-Fran\c cois C\^ot\'e for generously sharing
the C source code of his y-check solver, which greatly informed the
design of our implementation.

\section*{Data and code availability}

All instances used in this study---the CLASS, Type~I/II, ZDF-split,
and zero-waste conversions, together with the deterministic
generation and conversion protocols---and the source code of all our
components (formulations, LP-PC and the bound portfolio, BendM
master and cuts, and the matheuristic package) are publicly
released at
\url{https://github.com/HyunwooLee0429/GMSPP}, together with the
per-instance result files underlying every table of
Section~\ref{sec:experiments}. The provenance of the y-check
component is documented in Section~\ref{sec:bendm}.

\bibliography{refs}

\appendix
\section{Zero-waste family: stress test and endgame ablation}
\label{app:zerowaste}

The zero-waste instances of \citet{burke2004} are excluded from the
main comparisons for the reasons given in
Section~\ref{sec:instances}; here they serve two calibration roles.

\begin{table}[h!]
\centering\small
\caption{Zero-waste family: endgame on/off ablation (paired runs,
identical seed) and true-gap calibration.  $\Delta$ is the objective
change from disabling the endgame, averaged over both objectives and
$m\in\{2,3\}$ ($12$ pairs per row); positive would favor the
endgame.  All differences are within seed noise, and the endgame
improved no strip in any run.  The last column is the mean
\emph{true} optimality gap of the package under total area
($T=300$\,s), computable exactly here because the area bound equals
the optimum on zero-waste instances.}
\label{tab:nzw}
\begin{tabular}{lrrrrrr}
\toprule
 & & \multicolumn{2}{c}{$T=60$\,s} & \multicolumn{2}{c}{$T=300$\,s}
 & true gap\,\%\\
\cmidrule(lr){3-4}\cmidrule(lr){5-6}
Class & $n$ & mean $\Delta\%$ & max $|\Delta|\%$ &
mean $\Delta\%$ & max $|\Delta|\%$ & (total area)\\
\midrule
N1 & 17 & -0.18 & 2.70 & -0.34 & 1.39 & 7.25 \\
N2 & 25 & -0.02 & 1.94 & -0.09 & 0.47 & 7.54 \\
N3 & 29 & -0.11 & 1.12 & +0.00 & 0.00 & 6.55 \\
N4 & 49 & +0.17 & 1.82 & +0.12 & 1.12 & 6.08 \\
N5 & 73 & -0.00 & 2.04 & -0.00 & 0.38 & 5.43 \\
N6 & 97 & +0.26 & 1.38 & -0.13 & 2.08 & 5.15 \\
N7 & 197 & +0.13 & 1.53 & +0.05 & 0.96 & 3.24 \\
\bottomrule
\end{tabular}
\end{table}

First, an ablation of the exact endgame
(Table~\ref{tab:nzw}). On paired runs with identical seeds, at both
a short ($60$\,s) and the standard ($300$\,s) budget, disabling the
endgame changes the final objective only within seed noise (mean
effects of $\pm 0.1\%$ against a $10$-seed coefficient of variation
of $0.3$--$0.5\%$), and the endgame improved no strip in any of the
$336$ runs. Together with the probes of
Section~\ref{sec:ablation}, this settles the attribution question:
Phases~1--2 close the packing level heuristically on every family
we test, and the endgame's operational value is the per-strip
optimality certificate and the guarantee that this closure does not
depend on the packer---not net improvement.

Second, a true-gap calibration. On zero-waste instances the area
bound coincides with the optimum under the total-area objective, so
the certified gap \emph{is} the true optimality gap. The last
column of Table~\ref{tab:nzw} shows the package's true gaps of
$3$--$8\%$ on this jigsaw-like family, decreasing with size and
largest on the smallest, most puzzle-like instances---the hardest
regime for neighborhood search, as anticipated in
Section~\ref{sec:conclusion}.

\end{document}